\documentclass{amsart}
\usepackage{eurosym}
\usepackage{amsfonts}
\usepackage{amssymb}
\usepackage{graphicx}
\usepackage{amsmath}

\newtheorem{theorem}{Theorem}

\newtheorem{corollary}[theorem]{Corollary}

\newtheorem{definition}[theorem]{Definition}

\newtheorem{lemma}[theorem]{Lemma}

\newtheorem{proposition}[theorem]{Proposition}
\newtheorem{remark}[theorem]{Remark}

\begin{document}
\title{Decomposition of Singular Matrices into Idempotents}
\author{Adel Alahmadi, S. K. Jain, and Andr\'e Leroy}
\address{Adel Alahmedi\\
Department of Mathematics\\
King abdulaziz University\\
Jeddah, SA\\
Email:adelnife2@yahoo.com;\\
S. K. Jain\\
Department of Mathematics\\
King Andulaziz University Jeddah, SA,and\\
Ohio University, USA\\
Email:jain@ohio.edu\\
Andre Leroy\\
Facult\'e Jean Perrin\\
Universit\'e d'Artois\\
Lens, France\\
Email:andre.leroy55@gmail.com}
\maketitle

\begin{abstract}
In this paper we provide concrete constructions of idempotents to represent
typical singular matrices over a given ring as a product of idempotents and
apply these factorizations for proving our main results. 

We generalize works due to Laffey (
Products of idempotent matrices. Linear Multilinear A. 1983) and Rao (Products of idempotent matrices. Linear Algebra Appl. 2009) to noncommutative setting and
fill in the gaps in the original proof of Rao's main theorems.  We also consider singular matrices over
B\'ezout domains as to when such a matrix is a product of idempotent
matrices. 
\end{abstract}

\bigskip

\section{Introduction and definitions}

It was shown by Howie \cite{Ho} that every mapping from a finite set $X$ to
itself with image of cardinality $\leq $ $card X\,-1$\ is a product of
idempotent mappings. Erd\"{o}s \cite{E} (cf. also [14])
showed that every singular square
matrix over a field can be expressed as a product of idempotent matrices and
this was generalized by several authors to certain classes of rings, in
particular, to division rings and euclidean domains \cite{L}. Turning to
singular elements let us mention two results: Rao \cite{Bh} characterized,
via continued fractions, singular matrices over a commutative PID that can
be decomposed as a product of idempotent matrices and Hannah-O'Meara \cite%
{Ha} showed, among other results, that for a right self-injective regular
ring $R$, an element $a$ is a product of idempotents if and only if $%
Rr.ann(a)=l.ann(a)R$= $R(1-a)R$.

The purpose of this paper is to provide concrete constructions of
idempotents to represent typical singular matrices over a given ring as a
product of idempotents and to apply these factorizations for proving our
main results. Proposition \ref{product of idempotent matrices E=AB with BA=1}
and Theorem \ref{Elementary Hermite and IP_2 implies IP} fill in the gaps in
Rao's proof of a decomposition of singular matrices over principal ideal
domains (cf.~\cite{Bh}, Theorems 5 and 7), and simultaneously generalize
these results. We show that over a local ring $R$ (not necessarily
commutative), if every $2\times 2$ matrix $A$ with $r.ann(A)\neq 0$ is a product of idempotent matrices, then $R$ must be a domain (Theorem %
\ref{Local ring with IP are domains}). We prove the existence of a
decomposition into product of idempotents for any matrix $A$ with $%
l.ann(A)\neq 0,$ over a local domain (not necessarily commutative) with
Jacobson radical $J(R)=gR$ such that $\cap_n J(R)^n=0$ (Theorem \ref{local
ring with principal generator for J satisfies IP2}).

Let $R$ be a B\'ezout domain such that every $2\times 2$ singular matrix is
a product of idempotent matrices. Theorem \ref{Elementary Hermite and IP_2
implies IP} shows that if every $2\times 2$ invertible matrix over $R$ is a
product of elementary matrices and diagonal matrices with invertible
diagonal entries then every $n \times n$ singular matrix is a product of
idempotent matrices;

\noindent The converse of Theorem \ref{Elementary Hermite and IP_2 implies
IP} is true for commutative B\'ezout domain, that is, if every $n\times n$
singular matrix over such a domain is a product of idempotent matrices then
every $2\times 2$ invertible matrix is a product of elementary matrices and
diagonal matrices with invertible diagonal entries (Corollary \ref%
{Commutative Bezout and IP2 implies elementary Hermite}). Finally, Theorem %
\ref{Injective modules such that endo has the IP property} studies the
condition when each right singular element of the endomorphism ring of an
injective module is a product of projections.  This shows, in particular, that
each linear transformation of a vector space, which is right singular, is a
product of projections if and only if the vector space is finite-dimensional.


Let us now give the main definitions and fix our terminology.

All rings considered are nonzero rings with an identity element denoted by $%
1 $, and need not be commutative. A ring $R$ is called a local ring if it
has a unique maximal right ideal (equivalently, unique maximal left ideal).
For example, power series ring $F[[x]]$ over a field $F$ and the
localization $Z_{(p)}$ of the ring of integers $Z$ are local rings. A ring $%
R $ is called projective-free if each finitely generated projective right
(equivalently, left) module is free of unique rank. Every local ring is
projective-free. A ring $R$ is a principal right (left) ideal ring if each
right (left) ideal is principal. A right $R$-module $M$ is called injective
if every $R$-homomorphism from a right ideal of $R$ to $M$ can be extended
to an $R$-homomorphism from $R$ to $M.$ Clearly, every vector space over a
field is injective. A ring $R$ is called right self-injective if it is
injective as a right $R $-module. A ring $R$ is called von Neumann regular
if for each element $a\in R$, there exists an element $x\in R$ such that $%
axa=a$. A ring $R$ is called unit regular if for each element $a\in R$,
there exists an invertible element $u$ such that $aua=a $. A ring $R$ is
called Dedekind finite if for all $a,b\in R $, $ab=1$ implies $ba=1.$

An $n\times n$ matrix is called elementary if it is of the form $%
I_{n}+ce_{ij}$, $c\in R$ with $i\neq j$. 
A ring $R$ has a stable range $1$ if for any $a,b\in R$ with $aR+bR=R,$
there exists $x\in R$ such that $a+bx\in U(R)$, where $U(R)$ is the set of
invertible elements of $R$. A ring $R$ is right (left) B\'{e}zout if any
finitely generated right (left) ideal of $R$ is principal. Hermite rings
have been defined differently by different authors in the literature.
Following Kaplansky, we call $R$ to be a right (left) Hermite ring, if for
any two elements $a,b\in R$ there exists a $2\times 2$ invertible matrix $P$
and an element $d\in R$ such that $(a,b)P=(d,0)$ ($P(a,b)^{t}=(d,0)^{t} $).
Lam (\cite{La2}, section I, 4) calls this ring as K-Hermite ring. By a
Hermite (B\'ezout) ring we mean a ring which is both right and left Hermite
(B\'ezout). Amitsur showed that a ring $R$ is a right (left) Hermite domain
if and only if $R$ is a right (left) B\'{e}zout domain. Theorem \ref{Hermite
is Bezout} in this paper provides an alternative proof of Amitsur's theorem.

A ring $R$ is $GE_{2}$ if any invertible $2 \times 2$ matrix is a product of
elementary matrices and diagonal matrices with invertible diagonal entries.

A right unimodular row is a row $(a_{1},\dots ,a_{n})\in R^{n}$ with the
condition $\sum_{i=1}^{n}a_{i}R=R$. A right unimodular row is completable if
it is a row (equivalently, the bottom row) of an invertible matrix.

An element $a$ in a ring $R$ will be called right (left) singular if $r.ann$ 
$(a)$ $\neq 0$ $(l.ann$ $(a)$ $\neq 0$). An element is singular if it is
both left and right singular. $U(R)$ will denote the set of invertible
elements of a ring $R$. $M_{n\times m}(R)$ stands for the set of $n\times m$
matrices over the ring $R$. The ring of $n\times n$ matrices over $R$ will
be denoted by $M_n(R)$. The group of $n\times n$ invertible matrices over $R$
is denoted by $GL_n(R)$.

\section{Preliminaries.}

We begin with an elementary lemma which works like our reference table for
the proofs of our results. Note that one can obtain additional
factorizations from a given factorization into idempotent matrices by taking
conjugations.

\begin{lemma}
(Table of factorizations) \label{table of factorizations 2 by 2} Let $R$ be
any ring and let $a,b,c\in R$. Then

\begin{enumerate}
\item[(a)] $\ \ \ \left( 
\begin{array}{cc}
a & 0 \\ 
0 & 0%
\end{array}%
\right) =\left( 
\begin{array}{cc}
1 & a \\ 
0 & 0%
\end{array}%
\right) \left( 
\begin{array}{cc}
0 & 0 \\ 
0 & 1%
\end{array}%
\right) \left( 
\begin{array}{cc}
1 & 0 \\ 
1 & 0%
\end{array}%
\right),$

\item[($a^{\prime}$)] $\ \ \left( 
\begin{array}{cc}
0 & 0 \\ 
a & 0%
\end{array}%
\right) =\left( 
\begin{array}{cc}
0 & 0 \\ 
a & 1%
\end{array}%
\right) \left( 
\begin{array}{cc}
1 & 0 \\ 
0 & 0%
\end{array}%
\right) \left( 
\begin{array}{cc}
1 & 0 \\ 
1 & 0%
\end{array}%
\right) $,

\item[(b)] $\left( 
\begin{array}{cc}
a & ac \\ 
0 & 0%
\end{array}%
\right) =\left( 
\begin{array}{cc}
1 & a \\ 
0 & 0%
\end{array}%
\right) \left( 
\begin{array}{cc}
0 & 0 \\ 
0 & 1%
\end{array}%
\right) \left( 
\begin{array}{cc}
1 & 0 \\ 
1 & 0%
\end{array}%
\right) \left( 
\begin{array}{cc}
1 & c \\ 
0 & 0%
\end{array}%
\right) $,

\item[($b^{\prime}$)] $\left( 
\begin{array}{cc}
a & 0 \\ 
ca & 0%
\end{array}%
\right) =\left( 
\begin{array}{cc}
1 & 0 \\ 
c & 0%
\end{array}%
\right) \left( 
\begin{array}{cc}
1 & 1 \\ 
0 & 0%
\end{array}%
\right) \left( 
\begin{array}{cc}
0 & 0 \\ 
0 & 1%
\end{array}%
\right) \left( 
\begin{array}{cc}
1 & 0 \\ 
a & 0%
\end{array}%
\right) $,

\item[(c)] $\left( 
\begin{array}{cc}
ac & a \\ 
0 & 0%
\end{array}%
\right) =\left( 
\begin{array}{cc}
1 & a \\ 
0 & 0%
\end{array}%
\right) \left( 
\begin{array}{cc}
0 & 0 \\ 
c & 1%
\end{array}%
\right) $,

\item[($c^{\prime}$)] $\left( 
\begin{array}{cc}
ca & 0 \\ 
a & 0%
\end{array}%
\right) =\left( 
\begin{array}{cc}
0 & c \\ 
0 & 1%
\end{array}%
\right) \left( 
\begin{array}{cc}
1 & 0 \\ 
a & 0%
\end{array}%
\right) $,

\item[(d)] with $b\in U(R)$, $%
\begin{pmatrix}
a & b \\ 
0 & 0%
\end{pmatrix}%
=%
\begin{pmatrix}
b(b^{-1}a) & b \\ 
0 & 0%
\end{pmatrix}%
$ can be factorized as in (c) and $%
\begin{pmatrix}
a & 0 \\ 
b & 0%
\end{pmatrix}%
=%
\begin{pmatrix}
(ab^{-1})b & 0 \\ 
b & 0%
\end{pmatrix}%
$ can be factorized as in ($c^{\prime }$),

\item[(e)] with $a\in U(R)$, $%
\begin{pmatrix}
a & b \\ 
0 & 0%
\end{pmatrix}%
=%
\begin{pmatrix}
a & a(a^{-1}b) \\ 
0 & 0%
\end{pmatrix}%
$ can be factorized as in (b) and $%
\begin{pmatrix}
a & 0 \\ 
b & 0%
\end{pmatrix}%
=%
\begin{pmatrix}
a & 0 \\ 
(ba^{-1})a & 0%
\end{pmatrix}%
$ can be factorized as in ($b^{\prime }$).
\end{enumerate}
\end{lemma}

\bigskip

In the next lemma, we consider factorizations of $n\times n$ matrices:

\begin{lemma}
\label{Factorizations of 3 by 3 matrices} Let $R$ be any ring and $A\in
M_2(R)$ be either

\begin{enumerate}
\item[(a)] an elementary matrix,

\item[(b)] $%
\begin{pmatrix}
0 & 1 \\ 
1 & 0%
\end{pmatrix}%
$,

\item[(c)] a diagonal matrix,

\item[(d)] $%
\begin{pmatrix}
a & b \\ 
0 & 0%
\end{pmatrix}%
$ or $%
\begin{pmatrix}
a & 0 \\ 
b & 0%
\end{pmatrix}%
,\; a,b\in R.$
\end{enumerate}

Then, for $n\ge 3$, the $n\times n$ matrix $%
\begin{pmatrix}
A & 0 \\ 
0 & 0%
\end{pmatrix}%
$ is a product of idempotent matrices, where zero blocs are of appropriate
sizes.
\end{lemma}

\begin{proof}
We will treat the case when $n=3$. The general case is similar.

(a) Let us, for instance, choose an elementary matrix $A= I_2+ae_{12}$.

$A=\left( 
\begin{array}{cc}
1 & a \\ 
0 & 1%
\end{array}%
\right) $. Then $%
\begin{pmatrix}
1 & a & 0 \\ 
0 & 1 & 0 \\ 
0 & 0 & 0%
\end{pmatrix}%
=%
\begin{pmatrix}
1 & 0 & 0 \\ 
0 & 1 & 0 \\ 
0 & 0 & 0%
\end{pmatrix}%
\begin{pmatrix}
1 & a & 1 \\ 
0 & 1 & 0 \\ 
0 & -a & 0%
\end{pmatrix}%
\begin{pmatrix}
1 & 0 & 0 \\ 
0 & 1 & 0 \\ 
0 & 0 & 0%
\end{pmatrix}%
$.

(b) $%
\begin{pmatrix}
0 & 1 & 0 \\ 
1 & 0 & 0 \\ 
0 & 0 & 0%
\end{pmatrix}%
=%
\begin{pmatrix}
1 & 0 & 0 \\ 
0 & 1 & 0 \\ 
0 & 0 & 0%
\end{pmatrix}
\begin{pmatrix}
0 & 1 & 1 \\ 
1 & 0 & -1 \\ 
-1 & 1 & 2%
\end{pmatrix}
\begin{pmatrix}
1 & 0 & 0 \\ 
0 & 1 & 0 \\ 
0 & 0 & 0%
\end{pmatrix}%
$.

(c) $%
\begin{pmatrix}
a & 0 & 0 \\ 
0 & b & 0 \\ 
0 & 0 & 0%
\end{pmatrix}%
= 
\begin{pmatrix}
1 & 0 & -1 \\ 
0 & 1 & 0 \\ 
0 & 0 & 0%
\end{pmatrix}
\begin{pmatrix}
1 & 0 & 0 \\ 
0 & 1 & 0 \\ 
1-a & 0 & 0%
\end{pmatrix}
\begin{pmatrix}
1 & 0 & 0 \\ 
0 & 1 & -1 \\ 
0 & 0 & 0%
\end{pmatrix}
\begin{pmatrix}
1 & 0 & 0 \\ 
0 & 1 & 0 \\ 
0 & 1-b & 0%
\end{pmatrix}
$.

(d) Let us consider the case $%
\begin{pmatrix}
a & 0 & 0 \\ 
b & 0 & 0 \\ 
0 & 0 & 0%
\end{pmatrix}%
=%
\begin{pmatrix}
1 & 0 & a \\ 
0 & 1 & b \\ 
0 & 0 & 0%
\end{pmatrix}%
\begin{pmatrix}
0 & 0 & 0 \\ 
0 & 0 & 0 \\ 
0 & 0 & 1%
\end{pmatrix}%
\begin{pmatrix}
1 & 0 & 0 \\ 
1 & 0 & 0 \\ 
1 & 0 & 0%
\end{pmatrix}%
$.
\end{proof}

\begin{lemma}
\label{I.P. property implies Dedekind finite} Let $R$ be a ring. If each
right (left) singular element is a product of idempotents, then $R$ is
Dedekind finite.
\end{lemma}

\begin{proof}
Let $a,b\in R$ be such that $ab=1$. Then $l.ann(a)=0$. If $r.ann(a)=0$ then $%
a(ba-1)=0$ implies $ba=1,$ and we are done. In case $r.ann(a)\neq 0$ then by
hypothesis $a$ is a product of idempotents. This implies that $l.ann(a)\neq
0 $, a contradiction. Therefore, $r.ann(a)=0$ and so as above $ba=1.$
\end{proof}

The following lemma is well-known (cf. \cite{K}, Theorem 7.1).

\begin{lemma}
\label{relations between definitons}

If $R$ is right (left) Hermite domain then each right (left) unimodular row
is completable. 
\end{lemma}

\begin{lemma}
\label{bottom rows zero copy(1)} Let $A\in M_{n}(R)$ be a square matrix with
coefficients in a right B\'{e}zout domain $R$. Let $0\ne u\in R^{n}$ be such
that $uA=0$. Then there exists an invertible matrix $P\in GL_{n}(R)$ such
that $PAP^{-1}$has its last row equal to zero.
\end{lemma}

\begin{proof}
By hypothesis, for some $u\in R^{n}$ $uA=0$. Since $R$ is a right B\'{e}zout
domain we may assume that the vector $u$ is right unimodular. Since right B%
\'{e}zout domain are right Hermite we know that there exists an invertible
matrix $P$ such that the last row of $P$ is the vector $u$. Of course, this
implies that the last row of $PA$ is the zero row and this is true as well
for the last row of $PAP^{-1}$.
\end{proof}

%

Next, we list some properties and results for rings with stable range $1$
which will be referred to in the proofs. Let us first mention a well-known
theorem by Vaserstein which shows that the notion of stable range is
left-right symmetric.

\begin{lemma}
\label{a+bx=du} Let $a,a^{\prime },b,b^{\prime },x,d\in R$ and $u\in U(R)$
be such that $a+bx=du$, $a=da^{\prime }$ and $b=db^{\prime }$. Then

\begin{enumerate}
\item[(a)] $%
\begin{pmatrix}
a & b \\ 
0 & 0%
\end{pmatrix}%
=E%
\begin{pmatrix}
1 & 0 \\ 
-x & 1%
\end{pmatrix}%
$, where $E$ is a product of idempotent matrices,

\item[(b)] There exists an invertible matrix $P\in M_{2}(R)$ such that 
\begin{equation*}
\begin{pmatrix}
a & b \\ 
-x & 1%
\end{pmatrix}%
P=%
\begin{pmatrix}
d & 0 \\ 
0 & 1%
\end{pmatrix}%
.
\end{equation*}
\end{enumerate}
\end{lemma}

\begin{proof}
\noindent (a) Indeed we have: $%
\begin{pmatrix}
a & b \\ 
0 & 0%
\end{pmatrix}%
=%
\begin{pmatrix}
d & 0 \\ 
0 & 0%
\end{pmatrix}%
\begin{pmatrix}
u & b^{\prime } \\ 
0 & 0%
\end{pmatrix}%
\begin{pmatrix}
1 & 0 \\ 
-x & 1%
\end{pmatrix}%
$ and the first two matrices on the right side are products of idempotent
matrices as shown in the table of factorizations given in Lemma \ref{table
of factorizations 2 by 2}.

\noindent (b) The matrix $P$ is given by $P=%
\begin{pmatrix}
u^{-1} & -u^{-1}b^{\prime } \\ 
xu^{-1} & 1-xu^{-1}b^{\prime }%
\end{pmatrix}%
$.
\end{proof}

Rings with stable range $1$ possess many properties. The next lemma mentions
two of them that are particularly relevant to our study.

\begin{lemma}
\label{consequences of stable range$1$} Let $R$ be a ring with stable range $%
1$. Then

\begin{enumerate}
\item[(a)] $R$ is $GE_{2}$, and

\item[(b)] any unimodular row $(a,b)$ is completable.
\end{enumerate}
\end{lemma}

\begin{proof}

\noindent (a) Let $A=%
\begin{pmatrix}
a & b \\ 
c & d%
\end{pmatrix}
$ be an invertible matrix with coefficients in $R$. We thus have, in
particular, that $aR + bR=R$ and the stable range $1$ hypothesis shows that
there exists $x\in R$ such that $a+bx=u\in U(R)$. Let us put $%
v:=d-(c+dx)u^{-1}b$. We then have 
\begin{equation*}
A=%
\begin{pmatrix}
b & u \\ 
d & c+dx%
\end{pmatrix}%
\begin{pmatrix}
-x & 1 \\ 
1 & 0%
\end{pmatrix}
= 
\begin{pmatrix}
u & 0 \\ 
c+dx & v%
\end{pmatrix}%
\begin{pmatrix}
u^{-1}b & 1 \\ 
1 & 0%
\end{pmatrix}%
\begin{pmatrix}
-x & 1 \\ 
1 & 0%
\end{pmatrix}%
.
\end{equation*}
Since $A$ is invertible, $v$ is a unit. This finally gives us 
\begin{equation*}
A=%
\begin{pmatrix}
u & 0 \\ 
0 & v%
\end{pmatrix}%
\begin{pmatrix}
1 & 0 \\ 
v^{-1}(c+dx) & 1%
\end{pmatrix}%
\begin{pmatrix}
u^{-1}b & 1 \\ 
1 & 0%
\end{pmatrix}%
\begin{pmatrix}
-x & 1 \\ 
1 & 0%
\end{pmatrix}%
,
\end{equation*}
as required.

\noindent (b) If $aR+bR=R$ then there exists $x\in R$ such that $a+bx=u\in
U(R)$. In this case Lemma \ref{a+bx=du} shows that the unimodular row $(a,b)$
is completable.
\end{proof}

\section{Local Rings}

Firstly, as a consequence of our table of factorizations in Lemma \ref{table
of factorizations 2 by 2}, we give a very simple proof of the celebrated
theorem that every singular matrix over a division ring is a product of
idempotent matrices. The proof given below is for a singular $2\times 2$
matrix over a division ring. However, as a consequence of Theorem \ref%
{Elementary Hermite and IP_2 implies IP}, the proposition holds for any $%
n\times n$ singular matrix.

\begin{proposition}
Every $2\times 2$ singular matrix over a division ring can be factorized as
a product of idempotent matrices.
\end{proposition}

\begin{proof}
Let $A=\left( 
\begin{array}{cc}
a & c \\ 
b & d%
\end{array}%
\right) $ be a singular matrix. Then the columns $\left( 
\begin{array}{c}
a \\ 
b%
\end{array}%
\right) $ and $\left( 
\begin{array}{c}
c \\ 
d%
\end{array}%
\right) $ are linearly dependent. Suppose $\left( 
\begin{array}{c}
c \\ 
d%
\end{array}%
\right) =\left( 
\begin{array}{c}
a \\ 
b%
\end{array}%
\right) \alpha .$

Then $\left( 
\begin{array}{cc}
a & a\alpha \\ 
b & b\alpha%
\end{array}%
\right) =\left( 
\begin{array}{cc}
a & 0 \\ 
b & 0%
\end{array}%
\right) \left( 
\begin{array}{cc}
1 & \alpha \\ 
0 & 0%
\end{array}%
\right) $. If $b=0$, then Lemma 1 gives factorization of the first factor
whereas the second factor is already an idempotent. If $b\neq 0,$ then one
can use Lemma \ref{table of factorizations 2 by 2} (d) to conclude the
result.
\end{proof}

\bigskip

Next, we show that if each right (resp. each left) singular matrix over a local ring $R$ is a
product of idempotent matrices then the ring $R$ must be a domain. Let us
recall that a local ring is projective-free. For an idempotent matrix $E\in
M_{n}(R)$, $n>1$, where $R$ is projective-free, there exist matrices $A\in
M_{n\times r}(R)$ and $B\in M_{r\times n}(R)$ with $r<n$ such that $E=AB $
and $BA=I_{r}$ (See Cohn \cite{C}, Proposition 0.4.7, p. 24).

\begin{theorem}
\label{Local ring with IP are domains} Let $R$ be a local ring such that
each right (resp. each left) singular $2\times 2$ matrix over $R$ can be expressed as a product of
idempotents. Then $R$ is a domain.
\end{theorem}

\begin{proof}
We assume that every right singular matrix is a product of idempotents.  Let $a\in R.$ Suppose $r.ann(a)\neq 0.$ Since the matrix 
\begin{equation*}
A=%
\begin{pmatrix}
a & 0 \\ 
0 & 1%
\end{pmatrix}%
\end{equation*}%
is right singular it can be expressed as a product of idempotent matrices, say $%
A=E_{1}\dots E_{n}$. Since $a$ belongs to the Jacobson radical of $R$, it
cannot be itself an idempotent and hence we must have $n>1$. The property of
idempotent matrices recalled in the paragraph preceding this theorem shows
that $A$ can be written as $A=P_{1}Q_{1}\dots P_{n}Q_{n}$ where $Q_i\in
M_{1\times 2}(R),P_i\in M_{2\times 1}(R)$ are such that $Q_iP_i=1$. Set $%
P_{1}=(\alpha ,\beta )^{t}$, $Q_{1}P_{2}Q_{2}\dots P_{n}=\gamma \in R$ and $%
Q_{n}=(\delta ,\epsilon )$. Then $a=\alpha \gamma \delta $, $0=\alpha \gamma
\epsilon $, $0=\beta \gamma \delta $, and $1=\beta \gamma \epsilon $. Let us
set $P_{n}=(x,y)^t$. Since $Q_{n}P_{n}=1$, we obtain $\delta x+\epsilon y=1$%
. Furthermore, 
\begin{equation*}
\begin{pmatrix}
ax \\ 
y%
\end{pmatrix}%
=%
\begin{pmatrix}
a & 0 \\ 
0 & 1%
\end{pmatrix}%
\begin{pmatrix}
x \\ 
y%
\end{pmatrix}%
=%
\begin{pmatrix}
\alpha \\ 
\beta%
\end{pmatrix}%
\gamma 
\begin{pmatrix}
\delta & \epsilon%
\end{pmatrix}%
\begin{pmatrix}
x \\ 
y%
\end{pmatrix}%
=%
\begin{pmatrix}
\alpha \gamma \\ 
\beta \gamma%
\end{pmatrix}%
.
\end{equation*}%
This leads to $ax=\alpha \gamma $ and $y=\beta \gamma $. We then easily get $%
1=\beta \gamma \epsilon =y\epsilon $ and since $R$ is Dedekind finite we
also have $\epsilon y=1$. This leads consecutively to $\delta x=0$, $%
ax=\alpha \gamma \delta x=0$, $\alpha \gamma =ax=0$ and finally $a=\alpha
\gamma \delta =0$, as desired.
\end{proof}

%

The following theorem gives sufficient conditions for singular $2 \times 2$
matrices over local rings to be a product of idempotents.

\begin{theorem}
\label{local ring with principal generator for J satisfies IP2} Let $R$ be a
local domain such that its radical $J(R)=gR$ with $\cap_n $($J(R))^{n}=0$.
Let $S$ be the $2\times 2$ matrix ring over $R$. Then each matrix $A\in S$
with $l.ann(A)\neq 0$ is a product of idempotent matrices.
\end{theorem}

\begin{proof}
Since $J=gR$ \ with $\cap _{n}$($J(R))^{n}=0,$ we note that for any nonzero
elements $x,y\in R$ there exist positive integers $n,l$ such that $x=g^{n}u$
and $y=g^{l}v$, for some invertible elements $u,v\in U(R)$, where $U(R)$
denotes the set of invertible elements in $R$. If $n\geq l$ we can write $%
x=yc$ with $c:=v^{-1}g^{n-l}u$. Clearly, $c\neq 0$. Since $l.ann(A)\neq 0,$
we can assume that there exists $(x,y)\neq (0,0)$ such that $(x,y)A=(0,0)$.
Furthermore, since $x=yc, \, y\neq 0$ and $R$ is a domain, we have $%
(c,1)A=(0,0)$. This shows that $UA$ has bottom row zero where $U=%
\begin{pmatrix}
1 & 0 \\ 
c & 1%
\end{pmatrix}%
$ and so does the matrix $UAU^{-1}$. Since for every pair $(x,y)\neq (0,0)$,
one of them is a multiple of the other by invoking Lemma \ref{table of
factorizations 2 by 2} (b), we obtain that $A=U^{-1}E_{1}...E_{k}U,$ where $%
E_{i}$ are idempotents and hence $%
A=(U^{-1}E_{1}U)(U^{-1}E_{2}U)...(U^{-1}E_{k}U)$ is a product of idempotents.
\end{proof}

\begin{remark}
\textrm{If the matrix $A$ is such that 
$r.ann(A)\neq 0$ then the same proof will hold if we assume $J=Rg$ and $\cap
_{i\geq 0}J(R)^{^{i}}=0$. }
\end{remark}

\section{Construction of Idempotents and Representation of singular matrices}

The following lemma completes our "table" of Lemma 1 in an interesting way.
The lemma proved below provides a further useful tool while working with
idempotent matrices over a projective-free ring.

\begin{lemma}
\label{idempotent matrices from ab+cd=1} Let $a,b,c,d$ be elements in a ring 
$R.$

\begin{enumerate}
\item[(a)] If $ca+db=1$, then the matrix 
\begin{equation*}
E=%
\begin{pmatrix}
ac & ad \\ 
bc & bd%
\end{pmatrix}%
\end{equation*}%
is an idempotent matrix. If $R$ is a domain and the matrix $E$ is nonzero,
then the converse is also true.

\item[(b)] The matrix $%
\begin{pmatrix}
ab+u & a \\ 
0 & 0%
\end{pmatrix}%
$, $u$ a unit, is a product of idempotent matrices.

\item[(c)] If there exists $x\in R$ such that $a+bx\in U(R)$ then the matrix 
\begin{equation*}
\begin{pmatrix}
a & b \\ 
0 & 0%
\end{pmatrix}%
\end{equation*}%
is a product of idempotent matrices. 
\end{enumerate}
\end{lemma}

\begin{proof}
\noindent (a) This is easily checked.

\noindent (b) $%
\begin{pmatrix}
ab+u & a \\ 
0 & 0%
\end{pmatrix}%
=%
\begin{pmatrix}
u & 0 \\ 
0 & 0%
\end{pmatrix}%
\begin{pmatrix}
u^{-1}ab+1 & u^{-1}a \\ 
-b(u^{-1}ab+1) & -bu^{-1}a%
\end{pmatrix}%
. $

\noindent (c) By hypothesis, there exist $x\in R$ and $u\in U(R)$ such that $%
a+bx=u$. Hence $va=vb(-x)+1$ where $v=u^{-1}$. Using our previous table the
conclusion follows since one can write $%
\begin{pmatrix}
a & b \\ 
0 & 0%
\end{pmatrix}%
=%
\begin{pmatrix}
u & 0 \\ 
0 & 0%
\end{pmatrix}%
\begin{pmatrix}
va & vb \\ 
0 & 0%
\end{pmatrix}%
. $ Statement (b) above now yields the result.
\end{proof}

%

\begin{remark}
\textrm{The form of the $2\times 2$ idempotent matrix that appears in Lemma %
\ref{idempotent matrices from ab+cd=1} (a) is the only kind to consider in
the case when the ring $R$ is projective-free. Indeed in this case any $%
2\times2 $ idempotent matrix $A$ can be written as $A=%
\begin{pmatrix}
a \\ 
b%
\end{pmatrix}%
\begin{pmatrix}
c & d%
\end{pmatrix}%
$ with the condition that $ca+db=1$ (cf. the comments before Theorem \ref%
{Local ring with IP are domains}). }
\end{remark}

\smallskip

In view of this remark we look at the representation of a singular $2 \times
2$ matrix as product of idempotent matrices of the form $PQ^t$ where $P$ and 
$Q$ are columns vectors such $Q^tP=1.$

The next proposition translates the decomposition of a singular $2 \times 2$
matrix into a product of idempotents in terms of a family of equations. This
generalizes Rao's theorem (\cite{Bh}, Theorem 5) to noncommutative domains
and at the same time fills in the gaps in his original arguments (cf. \cite%
{Bh2}).

\begin{proposition}
\label{product of idempotent matrices E=AB with BA=1} Let $a,b$ be nonzero
elements in a domain $R$ such that $aR+bR=R$. Then the following are
equivalent:

\begin{enumerate}
\item[(i)] There exist an integer $n>0$ and elements $a_i,b_i,c_i,d_i\in R$, 
$i=1,\dots , n$ such that $a_1=c_1=1,\, b_1=0,\, c_n=a,\, d_n=b$,\, $%
c_ia_i+d_ib_i=1$, $1\le i \le n$ and $c_ia_{i+1}+d_ib_{i+1}=1$, $1\le i \le
n-1.$

\item[(ii)] There exist an integer $n>0$ and elements $a_i,b_i,c_i,d_i\in
R,\; 1\le i\le n,$ such that the matrix $%
\begin{pmatrix}
a & b \\ 
0 & 0%
\end{pmatrix}
$ can be written as a product $E_1\dots E_n$ of idempotent matrices $%
E_i^2=E_i$, where $E_i= 
\begin{pmatrix}
a_ic_i & a_id_i \\ 
b_ic_i & b_id_i%
\end{pmatrix}%
=%
\begin{pmatrix}
a_i \\ 
b_i%
\end{pmatrix}%
(c_i,d_i). $
\end{enumerate}

.
\end{proposition}

\begin{proof}
$(i)\Rightarrow (ii)$: Lemma \ref{idempotent matrices from ab+cd=1} (a)
shows that for $1\le i \le n$, the matrix $E_i=%
\begin{pmatrix}
a_ic_i & a_id_i \\ 
b_ic_i & b_id_i%
\end{pmatrix}
$ is an idempotent. Moreover, we have 
\begin{equation*}
E_1\cdots E_n=%
\begin{pmatrix}
a_1 \\ 
b_1%
\end{pmatrix}%
(c_1,d_1)%
\begin{pmatrix}
a_2 \\ 
b_2%
\end{pmatrix}
(c_2,d_2) \cdots 
\begin{pmatrix}
a_n \\ 
b_n%
\end{pmatrix}
(c_n,d_n),
\end{equation*}
and since $c_ia_{i+1}+d_ib_{i+1}=1$, $1\le i \le n-1$ and $a_1=1,\, b_1=0,\,
c_n=a,\, d_n=b$, we obtain $E_1E_2\cdots E_n=%
\begin{pmatrix}
a & b \\ 
0 & 0%
\end{pmatrix}%
$.

\noindent $(ii)\Rightarrow (i)$: We will construct elements $a_{i}^{\prime
},b_{i}^{\prime },c_{i}^{\prime },d_{i}^{\prime }$ satisfying the conditions
stated in $(i)$. Since $R$ is a domain and $E_{i}\neq 0$, Lemma \ref%
{idempotent matrices from ab+cd=1} (a) shows that for any $1\leq i\leq n$ we
have $c_{i}a_{i}+d_{i}b_{i}=1$. We can thus write 
\begin{equation*}
\begin{pmatrix}
a & b \\ 
0 & 0%
\end{pmatrix}%
=%
\begin{pmatrix}
a_{1} \\ 
b_{1}%
\end{pmatrix}%
(c_{1},d_{1})\cdots 
\begin{pmatrix}
a_{n} \\ 
b_{n}%
\end{pmatrix}%
(c_{n},d_{n}),\mathrm{with}\;c_{i}a_{i}+d_{i}b_{i}=1, \;\,1\leq i\leq n.
\end{equation*}%
If $s$ stands for the product $s:=%
\prod_{i=1}^{n-1}(c_{i}a_{i+1}+d_{i}b_{i+1})$, then we have $a=a_{1}sc_{n}$ $%
b=a_{1}sd_{n}$, $b_{1}sc_{n}=0$ and $b_{1}sd_{n}=0$. Since $R$ is a domain,
we easily get $b_{1}=0$ and $a_{1}c_{1}=1=c_{1}a_{1}$. Thus $E_{1}=%
\begin{pmatrix}
1 & a_1d_{1} \\ 
0 & 0%
\end{pmatrix}
$. We set $a_{1}^{\prime }=c_{1}^{\prime }=1,\; b_{1}^{\prime }=0,\;
d_{1}^{\prime }=a_1d_{1}$. Then we have 
\begin{equation*}
\begin{pmatrix}
a & b \\ 
0 & 0%
\end{pmatrix}%
=%
\begin{pmatrix}
1 \\ 
0%
\end{pmatrix}%
(1,d_{1}^{\prime })%
\begin{pmatrix}
a_{2} \\ 
b_{2}%
\end{pmatrix}%
\cdots 
\begin{pmatrix}
a_{n} \\ 
b_{n}%
\end{pmatrix}%
(c_{n},d_{n}), \,\mathrm{with}\;c_{i}a_{i}+d_{i}b_{i}=1\;\mathrm{for}\,1\leq
i\leq n.
\end{equation*}%
By comparing the entries on both sides we get $a=rc_{n},b=rd_{n}$, where $%
r:=(a_{2}+d^{\prime }_{1}b_{2})\prod_{i=2}^{n-1}(c_{i}a_{i+1}+d_{i}b_{i+1})$%
. By hypothesis, there exist $x,y\in R$ such that $ax+by=1$. This implies $%
rc_{n}x+rd_{n}y=1$. This shows that $r\in U(R)$. Set $u_{1}=(a_{2}+d^{\prime
}_{1}b_{2})\in U(R)$, $a_{2}^{\prime }=a_{2}u_{1}^{-1},\, b_{2}^{\prime
}=b_{2}u_{1}^{-1},\, c_{2}^{\prime }=u_{1}c_{2}$ and $d_{2}^{\prime
}=u_{1}d_{2}$. The matrix $E_{2}$ can be written $E_{2}=%
\begin{pmatrix}
a_{2}^{\prime } \\ 
b_{2}^{\prime }%
\end{pmatrix}%
(c_{2}^{\prime },d_{2}^{\prime })$. As per our definition $c_{1}^{\prime }=1$
and so we have $c_{1}^{\prime }a_{2}^{\prime }+d_{1}^{\prime }b_{2}^{\prime
}=a_{2}u_{1}^{-1}+d_{1}b_{2}u_{1}^{-1}=1$. Once again Lemma \ref{idempotent
matrices from ab+cd=1} (a) shows that $c_{2}^{\prime }a_{2}^{\prime
}+d_{2}^{\prime }b_{2}^{\prime }=1$ (this can of course, be checked
directly, as well). We then define $u_{2}:=(1,d_{1})E_{2}%
\begin{pmatrix}
a_{3} \\ 
b_{3}%
\end{pmatrix}%
=(1,d_{1})%
\begin{pmatrix}
a_{2}^{\prime } \\ 
b_{2}^{\prime }%
\end{pmatrix}%
(c_{2}^{\prime },d_{2}^{\prime })%
\begin{pmatrix}
a_{3} \\ 
b_{3}%
\end{pmatrix}%
=c_{2}^{\prime }a_{3}+d_{2}^{\prime }b_{3}=u_{1}(c_{2}a_{3}+d_{2}b_{3})\in
U(R)$ (since $u_{2}$ is a factor of $r$). Set $a_{3}^{\prime
}=a_{3}u_{2}^{-1},\, b_{3}^{\prime }=b_{3}u_{2}^{-1},\, c_{3}^{\prime
}=u_{2}c_{3}$ and $d_{3}^{\prime }=u_{2}d_{3}$. The matrix $E_{3}$ can be
written as $E_{3}=%
\begin{pmatrix}
a_{3}^{\prime } \\ 
b_{3}^{\prime }%
\end{pmatrix}%
(c_{3}^{\prime },d_{3}^{\prime })$. This gives $c_{3}^{\prime }a_{3}^{\prime
}+d_{3}^{\prime }b_{3}^{\prime }=1$ and $c_{2}^{\prime }a_{3}^{\prime
}+d_{2}^{\prime }b_{3}^{\prime }=u_{1}(c_{2}a_{3}+d_{2}b_{3})u_{2}^{-1}=1$.
We continue this process by defining $u_{3}:=(1,d_{1})E_{2}E_{3}%
\begin{pmatrix}
a_{4} \\ 
b_{4}%
\end{pmatrix}%
=c_{3}^{\prime }a_{4}+d_{3}^{\prime }b_{4}=u_{2}(c_{3}a_{4}+d_{3}b_{4})$, $%
a_{4}^{\prime }=a_{4}u_{3}^{-1},\, b_{4}^{\prime }=b_{4}u_{3}^{-1},\,
c_{4}^{\prime }=u_{3}c_{4}$, $d_{4}^{\prime }=u_{3}d_{4}$ and so on. In
general, we define for any $1\leq i\leq n-1$, $%
u_{i}:=u_{i-1}(c_{i}a_{i+1}+d_{i}b_{i+1})$, a factor of $r$ and hence $%
u_{i}\in U(R)$. Set $a_{i+1}^{\prime }:=a_{i+1}u_{i}^{-1},b_{i+1}^{\prime
}=b_{i+1}u_{i}^{-1},c_{i+1}^{\prime }:=u_{i}c_{i+1}$ and $d_{i+1}^{\prime
}:=u_{i}d_{i+1}$. The elements $a_{i}^{\prime },b_{i}^{\prime }c_{i}^{\prime
},d_{i}^{\prime }$, where $i\geq 2,$ together with $a_{1}^{\prime
}=1=c_{1}^{\prime },\, b_{1}^{\prime }=0,\, d_{1}^{\prime }=a_1d_{1}$ will
satisfy the required equalities. 
\end{proof}

\begin{corollary}
\label{product of idempotent matrices in projective-free domains} Let $%
a,b\in R$ be elements in a projective-free domain $R$ such that $aR+bR=R$.
Then the matrix $%
\begin{pmatrix}
a & b \\ 
0 & 0%
\end{pmatrix}%
$ is a product of idempotent matrices if and only if there exist an integer $%
n>0$ and elements $a_i,b_i,c_i,d_i\in R$, $i=1,\dots , n$ such that $%
a_1=c_1=1,\, b_1=0,\, c_n=a,\, d_n=b$, $c_ia_i+d_ib_i=1$, $1\le i \le n$ and 
$c_ia_{i+1}+d_ib_{i+1}=1$, $1\le i \le n-1$.
\end{corollary}

\begin{proof}
By the above proposition we know that the conditions mentioned in the
corollary are sufficient. Since any idempotent $2\times 2$ matrix with
coefficients in a projective-free domain is of the form $%
\begin{pmatrix}
a \\ 
b%
\end{pmatrix}%
(c,d)$ with $ca+db=1$, the implication $(ii) \Rightarrow (i)$ in above
Proposition shows that the conditions are also necessary.
\end{proof}

\section{Singular matrices over B\'{e}zout domains}

We first mention the classical facts that for any two elements $a,b$ in a
right B\'ezout domain $R$ both $aR +bR$ and $aR\cap bR$ are principal right
ideals and such a domain is a right Ore domain. The next theorem is due to
Amitsur \cite{Am}. We provide a different proof of this theorem. This proof
is inspired by results of Cohn (cf. \cite{C2}).

\begin{theorem}
\label{Hermite is Bezout} A domain $R$ is right Hermite if and only if it is
right B\'ezout.
\end{theorem}

\begin{proof}

Suppose $R$ is right Hermite. Then for $a,b\in R$ there exist $d\in R$ and $%
P\in GL_{2}(R)$ such that $(a,b)P=(d,0)$. Hence we have $dR\subseteq aR+bR$.
Since we also have $(d,0)P^{-1}=(a,b)$, we conclude that $aR+bR=dR\simeq R$.
This yields $R$ is right B\'ezout.

\noindent Conversely, suppose $R$ is right B\'{e}zout and so it is a right
Ore domain. Let $a,b\in R$. We first consider the case when $aR+bR=R$. We
know $aR\cap bR$ is a principal right ideal, say, $mR$. Let $x,y,u,v\in R$
be such that $ax+by=1$ and $au=m=bv$. We then obtain $a(xa-1)=-bya\in mR$
and so there exists $c\in R$ such that $xa-1=uc,\;vc=-ya$. Similarly, from $%
axb=b(1-yb)\in mR$, we get $d\in R$ such that $xb=ud$ and $1-yb=vd$. Let us
then consider the matrices 
\begin{equation*}
A:=%
\begin{pmatrix}
a & b \\ 
c & d%
\end{pmatrix}%
\quad \mathrm{and}\quad X:=%
\begin{pmatrix}
x & -u \\ 
y & v%
\end{pmatrix}%
.
\end{equation*}%
We can check that the above relations give $XA=I$ and so $AX=I$, since $R$
is embeddable in a division ring. This, in turn, leads to $(a,b)X=(1,0)$.

\noindent In the general case, we have $aR+bR=dR$. We can write $%
a=da^{\prime }$, $b=db^{\prime }$ and then since $R$ is a domain, $a^{\prime
}R+b^{\prime }R=R$. So $R$ is right B\'ezout as shown above.
\end{proof}

\begin{definition}
We say that a ring $R$ has the $IP_2$ property if every $2\times 2$ singular
matrix is a product of idempotent matrices. 
\end{definition}

Of course, every ring for which singular matrices are products of idempotent
matrices has $IP_{2}$. In particular, every commutative euclidean domain has
the $IP_{2}$ property as shown by Laffey (cf. \cite{L}).

\begin{lemma}
\label{Bezout and stable range $1$ implies IP2} A left (right) B\'ezout
domain with stable range $1$ has the $IP_{2}$ property.
\end{lemma}

\begin{proof}
Let $A\in M_2(R)$ be a singular matrix. By Lemma \ref{bottom rows zero
copy(1)} we may assume that the matrix has a bottom row equal to zero. Since
matrices of the form $%
\begin{pmatrix}
a & 0 \\ 
0 & 0%
\end{pmatrix}
$ are products of idempotent matrices (cf. Lemma \ref{table of
factorizations 2 by 2} ), we may further assume that the first row of $A$ is
unimodular. The hypothesis of stable range $1$ and Lemma \ref{idempotent
matrices from ab+cd=1} (c) show that $R$ has the $IP_2$ property.
\end{proof}


In the case of a commutative B\'{e}zout domain we can replace the stable range 
$1$ hypothesis by the $IP2$ property and still get strong conclusions as shown 
in Proposition \ref{link between product of
idempotents and Hermite} and Corollary \ref{Commutative Bezout and IP2
implies elementary Hermite}. Indeed proposition \ref{link between product of
idempotents and Hermite} provides a relationship between a decomposition of
a singular matrix into a product of idempotent matrices and B\'{e}zout
(equivalently, Hermite) domains.

\begin{proposition}
\label{link between product of idempotents and Hermite} Let $a,b$ be
elements in a commutative B\'{e}zout domain $R$ with $aR+bR=R$. Then the
following are equivalent.

\begin{enumerate}
\item[(i)] $%
\begin{pmatrix}
a & b \\ 
0 & 0%
\end{pmatrix}%
$ is a product of $n$ idempotent matrices.

\item[(ii)] There exist elements $r_0,r_1 \dots r_{2n-2} \in
R $ such that 
\begin{equation*}
(a,b)=(1,0)%
\begin{pmatrix}
r_{2n-2} & 1 \\ 
1 & 0%
\end{pmatrix}%
\begin{pmatrix}
r_{2n-3} & 1 \\ 
1 & 0%
\end{pmatrix}%
\cdots 
\begin{pmatrix}
r_0 & 1 \\ 
1 & 0%
\end{pmatrix}%
\begin{pmatrix}
0 & -1 \\ 
1 & 0%
\end{pmatrix}%
.
\end{equation*}
\end{enumerate}
\end{proposition}

\begin{proof}
$(i)\Rightarrow (ii)$: Since a right B\'{e}zout domain is projective-free,
Corollary \ref{product of idempotent matrices in projective-free domains}
shows that there exist elements $a_{i},b_{i},c_{i},d_{i}\in R$, $i=1,\dots
,n $ such that $a_{1}=c_{1}=1,\,b_{1}=0,\,c_{n}=a,\,d_{n}=b$, $%
c_{i}a_{i}+d_{i}b_{i}=1$ for $1\leq i\leq n$ and $%
c_{i}a_{i+1}+d_{i}b_{i+1}=1 $, $1\leq i\leq n-1$. Write 
\begin{equation*}
\begin{pmatrix}
a & b \\ 
0 & 0%
\end{pmatrix}%
=%
\begin{pmatrix}
a_{1} \\ 
b_{1}%
\end{pmatrix}%
(c_{1},d_{1})%
\begin{pmatrix}
a_{2} \\ 
b_{2}%
\end{pmatrix}%
(c_{2},d_{2})\cdots 
\begin{pmatrix}
a_{n} \\ 
b_{n}%
\end{pmatrix}%
(c_{n},d_{n}).
\end{equation*}%
Let us put $r _{2n-2}=c_{n}d_{n-1}-c_{n-1}d_{n}$ and $r_{2n-3}
=a_{n-1}b_{n}-a_{n}b_{n-1}$. We then write, successively, 
\begin{equation*}
\begin{pmatrix}
c_{n} & d_{n} \\ 
0 & 0%
\end{pmatrix}%
=%
\begin{pmatrix}
1 & 0 \\ 
0 & 0%
\end{pmatrix}%
\begin{pmatrix}
c_{n} & d_{n} \\ 
b_{n} & -a_{n}%
\end{pmatrix}%
=%
\begin{pmatrix}
1 & 0 \\ 
0 & 0%
\end{pmatrix}%
\begin{pmatrix}
r_{2n-2} & 1 \\ 
1 & 0%
\end{pmatrix}%
\begin{pmatrix}
b_{n} & -a_{n} \\ 
c_{n-1} & d_{n-1}%
\end{pmatrix}%
,
\end{equation*}%
and 
\begin{equation*}
\begin{pmatrix}
a & b \\ 
0 & 0%
\end{pmatrix}%
=%
\begin{pmatrix}
c_{n} & d_{n} \\ 
0 & 0%
\end{pmatrix}%
=%
\begin{pmatrix}
1 & 0 \\ 
0 & 0%
\end{pmatrix}%
\begin{pmatrix}
r_{2n-2} & 1 \\ 
1 & 0%
\end{pmatrix}%
\begin{pmatrix}
r_{2n-3} & 1 \\ 
1 & 0%
\end{pmatrix}%
\begin{pmatrix}
c_{n-1} & d_{n-1} \\ 
b_{n-1} & -a_{n-1}%
\end{pmatrix}%
.
\end{equation*}%
Continuing this process we put, for $1\leq i\leq n-1$,\thinspace\ $r_{2(n-i)}
=c_{n-i+1}d_{n-i}-c_{n-i}d_{n-i+1}$ and $r_{2(n-i)-1}=a_{n-i}b_{n-i+1}-a_{n-i+1}b_{n-i}$. 
With these notations one
gets: 
\begin{equation*}
\begin{pmatrix}
a & b \\ 
0 & 0%
\end{pmatrix}%
=%
\begin{pmatrix}
1 & 0 \\ 
0 & 0%
\end{pmatrix}%
\begin{pmatrix}
r_{2n-2} & 1 \\ 
1 & 0%
\end{pmatrix}%
\cdots 
\begin{pmatrix}
r_{1} & 1 \\ 
1 & 0%
\end{pmatrix}%
\begin{pmatrix}
c_{1} & d_{1} \\ 
b_{1} & -a_{1}%
\end{pmatrix}%
,
\end{equation*}%
where $c_{1}=1,\,b_{1}=0,\,a_{1}=1$. Hence 
\begin{equation*}
\begin{pmatrix}
c_{1} & d_{1} \\ 
b_{1} & -a_{1}%
\end{pmatrix}%
=%
\begin{pmatrix}
1 & d_{1} \\ 
0 & -1%
\end{pmatrix}%
=%
\begin{pmatrix}
-d_{1} & 1 \\ 
1 & 0%
\end{pmatrix}%
\begin{pmatrix}
0 & -1 \\ 
1 & 0%
\end{pmatrix}%
.
\end{equation*}%
This, finally, yields that 
\begin{equation*}
\begin{pmatrix}
a & b \\ 
0 & 0%
\end{pmatrix}%
=%
\begin{pmatrix}
1 & 0 \\ 
0 & 0%
\end{pmatrix}%
\begin{pmatrix}
r_{2n-2} & 1 \\ 
1 & 0%
\end{pmatrix}%
\cdots 
\begin{pmatrix}
r_{1} & 1 \\ 
1 & 0%
\end{pmatrix}%
\begin{pmatrix}
r_{0} & 1 \\ 
1 & 0%
\end{pmatrix}%
\begin{pmatrix}
0 & -1 \\ 
1 & 0%
\end{pmatrix}%
,
\end{equation*}
completing the proof with $r_{0}=-d_{1}$.

$(ii)\Rightarrow (i)$: We are given $2n-1$ elements $r_{i},\,0\leq
i\leq 2n-2$, and we want to produce $4n$ elements $a_{j},b_{j},c_{j},d_{j}$, 
$1\leq j\leq n$, satisfying the relations given in the Proposition \ref%
{product of idempotent matrices E=AB with BA=1}. Let us show how to retrace
our steps. First, we have $a_{1}=1=c_{1},\,b_{1}=0,\,d_{1}=-r_{0}$,
and $a_{1}c_{1}+d_{1}b_{1}=1$. Suppose that we have constructed $%
a_{j},b_{j},c_{j},d_{j}$ satisfying the necessary relations for all $1\leq
j\leq i$ and let us show how to construct $a_{i+1},b_{i+1},c_{i+1},d_{i+1}$.
We determine $a_{i+1}$ and $b_{i+1}$ via the system of equations:
\begin{eqnarray}
a_ib_{i+1}-b_ia_{i+1} & = r_{2i-1}, \\
d_ib_{i+1}+c_ia_{i+1} & =1.
\end{eqnarray}
Since $a_ic_i+d_ib_i=1$, the above system has a unique solution. To
determine $c_{i+1}$ and $d_{i+1}$ we use the following equations, 
\begin{eqnarray}
d_ic_{i+1}-c_id_{i+1} & = r_{2i}, \\
a_{i+1}c_{i+1}+b_{i+1}d_{i+1} & =1.
\end{eqnarray}
Then $d_ib_{i+1}+c_ia_{i+1}=1$ gives that the above system has a unique
solution.
\end{proof}

The next corollary gives another proof of Lemma 2 in Laffey's paper \cite{L}.

\begin{corollary}
\label{Euclidean rings have IP2} Let $R$ be a euclidean domain. Then $R$ has
the $IP_2$ property.
\end{corollary}

\begin{proof}
We have to show that any singular matrix $A\in M_2(R)$ is a product of
idempotent matrices. Lemma \ref{bottom rows zero copy(1)} shows that we may
assume $A$ is of the form $%
\begin{pmatrix}
a & b \\ 
0 & 0%
\end{pmatrix}%
.$ Let $d\in R$ be such that $aR+bR=dR$ and write $a=da^{\prime }$, $%
b=db^{\prime }$ for some $a^{\prime },b^{\prime }\in R$. We have 
\begin{equation*}
\begin{pmatrix}
a & b \\ 
0 & 0%
\end{pmatrix}%
=%
\begin{pmatrix}
d & 0 \\ 
0 & 0%
\end{pmatrix}%
\begin{pmatrix}
a^{\prime } & b^{\prime } \\ 
0 & 0%
\end{pmatrix}%
.
\end{equation*}
Since matrices $%
\begin{pmatrix}
d & 0 \\ 
0 & 0%
\end{pmatrix}%
$ are always product of idempotent matrices, we may assume without loss of
generality that $aR+bR=R$. The euclidean algorithm provides sequences of
elements $q_0,q_1,\dots,q_n,q_{n+1}$ and $r_0,r_1,\dots,r_n$ in $R$ such
that $-b=aq_0+r_0,\;a=r_0q_1+r_1,\;\dots,r_{n-2}=r_{n-1}q_n+1, \;
r_{n-1}=q_{n+1}$. We then have: 
\begin{equation*}
(-b,a)=(a,r_0)%
\begin{pmatrix}
q_0 & 1 \\ 
1 & 0%
\end{pmatrix}%
= (r_0,r_1)%
\begin{pmatrix}
q_1 & 1 \\ 
1 & 0%
\end{pmatrix}%
\begin{pmatrix}
q_0 & 1 \\ 
1 & 0%
\end{pmatrix}%
,
\end{equation*}
and finally 
\begin{equation*}
(-b,a)=(1,0)%
\begin{pmatrix}
q_{n+1} & 1 \\ 
1 & 0%
\end{pmatrix}
\begin{pmatrix}
q_n & 1 \\ 
1 & 0%
\end{pmatrix}%
\cdots 
\begin{pmatrix}
q_0 & 1 \\ 
1 & 0%
\end{pmatrix}%
.
\end{equation*}
Right multiplying this equality by the matrix $%
\begin{pmatrix}
0 & -1 \\ 
1 & 0%
\end{pmatrix}%
$ and using Proposition \ref{link between product of idempotents and Hermite}%
, we conclude that the matrix $%
\begin{pmatrix}
a & b \\ 
0 & 0%
\end{pmatrix}%
$ is a product of idempotent matrices, as required.
\end{proof}

\begin{corollary}
\label{Commutative Bezout and IP2 implies elementary Hermite} If $R$ is a
commutative B\'ezout domain with the $IP_2$ property then every $2\times 2$
invertible matrix is a product of elementary matrices and diagonal matrices
with invertible diagonal entries.
\end{corollary}

\begin{proof}

Let $A=%
\begin{pmatrix}
a & b \\ 
c & d%
\end{pmatrix}%
\in GL_2(R) $. We thus have $aR+bR=R$ and the $IP_2$ property shows that the
matrix $%
\begin{pmatrix}
a & b \\ 
0 & 0%
\end{pmatrix}
$ is a product of $n$ idempotent matrices, for some $n$. Proposition \ref%
{link between product of idempotents and Hermite} then shows that $%
(a,b)=(1,0)U$, where 
\begin{equation*}
U= 
\begin{pmatrix}
r_{2n-2} & 1 \\ 
1 & 0%
\end{pmatrix}%
\begin{pmatrix}
r_{2n-3} & 1 \\ 
1 & 0%
\end{pmatrix}%
\cdots 
\begin{pmatrix}
r_0 & 1 \\ 
1 & 0%
\end{pmatrix}%
\begin{pmatrix}
0 & -1 \\ 
1 & 0%
\end{pmatrix}
\in GL_2(R).
\end{equation*}
Since for $r \in R$ we have: 
\begin{equation*}
\begin{pmatrix}
r & 1 \\ 
1 & 0%
\end{pmatrix}%
=%
\begin{pmatrix}
1 & r \\ 
0 & 1%
\end{pmatrix}%
\begin{pmatrix}
1 & 0 \\ 
1 & 1%
\end{pmatrix}%
\begin{pmatrix}
1 & 0 \\ 
0 & -1%
\end{pmatrix}%
\begin{pmatrix}
1 & 1 \\ 
0 & 1%
\end{pmatrix}%
\begin{pmatrix}
1 & 0 \\ 
-1 & 1%
\end{pmatrix}%
,
\end{equation*}
we conclude that matrices of the form $%
\begin{pmatrix}
r & 1 \\ 
1 & 0%
\end{pmatrix}%
$ and hence also the matrix $U$ are products of elementary and diagonal
matrices with invertible diagonal entries. Let us write $AU^{-1}=%
\begin{pmatrix}
1 & 0 \\ 
c^{\prime } & d^{\prime }%
\end{pmatrix}
$ for some $c^{\prime },d^{\prime }\in R$. We then have $A=%
\begin{pmatrix}
1 & 0 \\ 
c^{\prime } & 1%
\end{pmatrix}%
\begin{pmatrix}
1 & 0 \\ 
0 & d^{\prime }%
\end{pmatrix}%
U$. This shows that $A$ is a product of elementary matrices and diagonal
matrices with invertible diagonal entries, as desired. 
\end{proof}

We say that a ring $R$ has the $IP$ property if every singular square matrix
over $R$ can be written as a product of idempotent matrices. Of course, the $%
IP$ property implies the $IP_2$ property. Theorem \ref{Elementary Hermite
and IP_2 implies IP} shows that in certain situations the converse is true,
that is: $IP_{2}$ property implies $IP$ property. The proof of this result
follows the pattern of Laffey's proof \cite{L}.

\begin{theorem}
\label{Elementary Hermite and IP_2 implies IP} Let $R$ be a B\'ezout domain
satisfying the $IP_{2}$ property. Then every singular matrix is a product of
idempotent matrices if $R$ has the $GE_{2}$ property. In particular, a
B\'ezout domain with stable range 1 has the $IP$ property.
\end{theorem}

\begin{proof}
Let $A\in M_{n}(R)$ be a singular matrix. Lemma \ref{bottom rows zero
copy(1)} shows that we may assume that the bottom row of $A$ is zero. Let us
write 
\begin{equation*}
A=%
\begin{pmatrix}
B & C \\ 
0 & 0%
\end{pmatrix}%
,
\end{equation*}%
where $B\in M_{n-1}(R)$ and other matrices are of appropriate sizes.

We now proceed by induction on $n$. The case $n=1$ is trivial since $R$ is a
domain. If $n=2$ this is the $IP_{2}$ property. Let $n\in \mathbb{N}$ be
such that $n>2$. Write 
\begin{equation*}
A=%
\begin{pmatrix}
I_{n-1} & C \\ 
0 & 0%
\end{pmatrix}%
\begin{pmatrix}
B & 0 \\ 
0 & 1%
\end{pmatrix}%
.
\end{equation*}

If $B$ is singular, we apply induction hypothesis on $B$ and we obtain $A$
as a product of idempotents. So let us assume, $B$ is nonsingular. Since $R$
is left B\'ezout and hence left Hermite, by invoking $GE_2$ we can find a
sequence of elementary matrices $E_{1},\dots ,E_{l}\in M_{n-1}(R)$ such that 
$D:=E_{1}\cdots E_{l}B$ is an upper triangular matrix. For $M\in M_{n-1}(R)$%
, we define $\widehat{M}\in M_n(R)$ by 
\begin{equation*}
\widehat{M}:=%
\begin{pmatrix}
M & 0 \\ 
0 & 0%
\end{pmatrix}%
\in M_{n}(R).
\end{equation*}
We then have 
\begin{equation*}
A=%
\begin{pmatrix}
E_1^{-1}\cdots E_l^{-1}D & C \\ 
0 & 0%
\end{pmatrix}
=\widehat{E_{1}^{-1}}\widehat{E_{2}^{-1}}\cdots \widehat{E_{l}^{-1}}%
\begin{pmatrix}
D & E_{l}\cdots E_{1}C \\ 
0 & 0%
\end{pmatrix}%
.
\end{equation*}
Lemma \ref{Factorizations of 3 by 3 matrices} shows that $\widehat{E_{1}^{-1}%
},\widehat{E_{2}^{-1}},\cdots , \widehat{E_{l}^{-1}}$ are products of
idempotent matrices. We thus have to show that the matrix 
\begin{equation*}
\begin{pmatrix}
D & E_{l}\cdots E_{1}C \\ 
0 & 0%
\end{pmatrix}%
= 
\begin{pmatrix}
d_{1} & \left( 
\begin{array}{cccc}
d_{2} & .. & .. & d_{n}%
\end{array}%
\right) \\ 
0 & D_{1}%
\end{pmatrix}%
\end{equation*}
is a product of idempotents. The last row of the triangular matrix $D_1$ is
zero and hence our induction hypothesis shows that there exist idempotent
matrices $Y_1,\dots, Y_m$ such that $D_1=Y_1\cdots Y_m$. We may assume $%
Y_m\ne I_{n-1}$ and, since $D_1Y_m=D_1$, we can write 
\begin{equation*}
\begin{pmatrix}
d_{1} & \left( 
\begin{array}{cccc}
d_{2} & .. & .. & d_{n}%
\end{array}%
\right) \\ 
0 & D_{1}%
\end{pmatrix}%
=%
\begin{pmatrix}
1 & 0 \\ 
0 & D_1%
\end{pmatrix}%
\begin{pmatrix}
d_1 & (d_2 \cdots d_n) \\ 
0 & Y_m%
\end{pmatrix}%
.
\end{equation*}
The first left factor matrix on the right hand side is a product of
idempotent matrices. Thus we only need show that the second factor matrix on
the right hand side is a product of idempotent matrices. Now, since $R$ is
projective-free we know that $Y_{m}$ is similar to a diagonal matrix with
only ones and zeros on the diagonal (cf. \cite{C}, Proposition 0.4.7.). We
claim that the number, say $h$, of ones on this diagonal is strictly
positive or, in other words, we claim that $Y_m\ne 0$. Indeed, if $Y_m=0$
then $D_1=D_1Y_m=0$ and a row of the matrix $D$ is zero (since $n\geq 3$).
Hence $D=E_1\cdots E_lB$ is singular, this implies that $B$ is singular, a
contradiction since $B$ is non singular. We are thus reduced to show that a
matrix of the form 
\begin{equation*}
\begin{pmatrix}
d_{1} & \left( 
\begin{array}{cccc}
d_{2} & .. & .. & d_{h+1}%
\end{array}%
\right) & ... & d_{n} \\ 
0 & I_{h} & 0 & 0 \\ 
.. & .. & .. & .. \\ 
0 & 0 & 0 & 0%
\end{pmatrix}%
,
\end{equation*}%
for some $h>0$ is a product of idempotent matrices. This matrix is similar
to the following: 
\begin{equation*}
\begin{pmatrix}
I_{h} & 0 & 0 & \cdots & 0 \\ 
(d_2,\dots,d_{h+1}) & d_{1} & d_{h+2} & \cdots & d_{n} \\ 
\cdots & \cdots & \cdots & \cdots & \cdots \\ 
0 & 0 & 0 & 0 & 0%
\end{pmatrix}%
.
\end{equation*}
Performing row elementary operations on the first $h+1$ rows we reduce the
above matrix to the following: 
\begin{equation*}
\begin{pmatrix}
I_{h} & 0 & 0 & \cdots & 0 \\ 
0 & d_{1} & d_{h+2} & \cdots & d_{n} \\ 
\cdots & \cdots & \cdots & \cdots & \cdots \\ 
0 & 0 & 0 & 0 & 0%
\end{pmatrix}%
=%
\begin{pmatrix}
I_h & 0 \\ 
0 & *%
\end{pmatrix}%
,
\end{equation*}
where the bloc matrix * is an $(n-h)\times (n-h)$ matrix with the last row
zero. Since these operations can be accomplished by multiplying on the left
by products of idempotent matrices the induction hypothesis applied to the
matrix * completes the proof.

The proof of the particular case is clear since Lemma \ref{consequences of
stable range$1$} and \ref{Bezout and stable range $1$ implies IP2} show that
stable range $1$ implies $GE_2$ and $IP_2$, respectively.
\end{proof}

As a special case of the above theorem, the following corollary, parts (a)
and (c), gives Laffey's theorem (cf.\,\cite{L}) and Rao's theorem (cf.\,\cite%
{Bh}, Theorem 2), respectively.

\begin{corollary}
\label{4 cases when $R$ has IP} Let $R$ be a domain which is any one of the
following types:

\begin{enumerate}
\item[(a)] a euclidean domain,

\item[(b)] a local domain such that its radical $J=Rg=gR$ with $\cap
Rg^{n}=0 $,

\item[(c)] a commutative principal ideal domain with $IP_2$, or

\item[(d)] a local B\'{e}zout domain.
\end{enumerate}

Then every singular matrix over $R$ is a product of idempotent matrices (in
other words, $R$ has the $IP$ property).
\end{corollary}

\begin{proof}
(a) It is clear that a euclidean domain is B\'ezout. On the other hand,
Corollaries \ref{Euclidean rings have IP2} and \ref{Commutative Bezout and
IP2 implies elementary Hermite} show that $R$ has the $IP_2$ as well as the $%
GE_2$ property. Therefore, by the above theorem $R$ has the $IP$ property.

(b) It is clear that a local ring has stable range $1$ and by our hypothesis 
$R$ is a valuation domain (as in the proof of Theorem \ref{local ring with
principal generator for J satisfies IP2}), and hence a B\'ezout domain. The
particular case mentioned in the above theorem yields the result. 

(c) This follows from the above theorem and from Corollary \ref{Commutative
Bezout and IP2 implies elementary Hermite}. 

(d) Since a local domain has stable range $1$, the result follows from the
particular case of the theorem.
\end{proof}

\section{Endomorphisms of Injective Modules}

Finally, we consider the endomorphism ring $S$ of an injective module $M$.
Recall, from the introduction, that an element $s\in S$ is called right singular if $rann(s)\ne 0$.  
We know that if any ring has the $IP$ property, then it need not be of
stable range $1$. However, for the endomorphism ring
of an injective (or even quasi-injective) module, we have the
following theorem. Its proof is straightforward.

\begin{theorem}
\label{Injective modules such that endo has the IP property} Let $M_{R}$ be
an injective module and $S=End(M_{R})$. If each right singular element $s\in
S$ can be expressed as a product of idempotents, then $S$ has stable range $%
1.$
\end{theorem}

\begin{proof}
Let $J=J(S)$ denote the Jacobson radical of $S=End(M_{R})$. Lemma \ref{I.P.
property implies Dedekind finite} shows that $S$ is Dedekind finite. It is a
folklore that $S/J$ is also Dedekind finite. We provide its proof for
reference only. For $x\in S$, let us write $\overline{x}:=x+J$. If $a,b\in S$
are such that $\overline{a}\overline{b}=\overline{1}$ then $1-ab\in J$ and
hence $1-(1-ab)=ab\in U(S)$, the set of units of $S$. Since $S$ is Dedekind
finite, we also have $ba\in U(S)$. Thus there exists $c\in S$ such that $%
bac=1$ and we get $\overline{bac}=\overline{1}$. Since $(\overline{b}%
\overline{a}-\overline{1})\overline{b}=\overline{0}$, we obtain by post
multiplying this by $\overline{ac},$ $\overline{b}\overline{a}=\overline{1}$%
, as desired. It is well-known that $S/J$ is a regular right self-injective
ring (cf. \cite{La}, Theorem 13.1). Because a von Neumann regular right
self-injective Dedekind finite ring is a unit regular ring (cf. \cite{G},
Theorem 9.17), it follows that $S/J(S)$ is a unit regular ring. This implies 
$S/J(S)$ has stable range $1$.

We now prove $S$ has stable range $1$. Let $aS+bS=S.$ Then ($%
a+J(S))(S+J(S))+(b+J(S))(S+J(S))=S+J(S).$ This gives $%
(a+J(S))+(b+J(S))(u+J(S))$ is invertible for some $u+J(S)$ in $S/J(S).$ This
implies that there exists $v\in S$ such that $(a+bu)v-1\in J(S)$ and hence $%
(a+bu)v$ is invertible. Since $S$ is Dedekind finite $a+bu$ is invertible.
This concludes the proof that $S$ is of stable range $1.$
\end{proof}

\begin{remark}
\textrm{Since the endomorphism ring of an infinite dimensional vector space
is not of stable range $1$, it follows that every right singular 
(equivalently non monomorphism) endomorphism can
be expressed as a product of projections if and only if the vector space is
of finite dimension.}
\end{remark}

\bigskip

\centerline {\bf Acknowledgement}

\smallskip 

1.We would like to thank K. P. S. Bhaskara Rao for several useful exchange
of emails regarding his paper, Products of Idempotent Matrices, LAA (2009).

2.This research is supported by King Abdulaziz University and, in
particular, by the Deanship of Scientific Research.

3.Andre Leroy would like to thank King Abdulaziz University for their
hospitality during his visit when this work was in progress.

\end{document}